\documentclass[11pt,a4paper,reqno]{amsart} 

\usepackage[utf8]{inputenc} 

\usepackage{graphicx} 

\usepackage{booktabs} 
\usepackage{array} 
\usepackage{paralist} 
\usepackage{verbatim} 
\usepackage{subfig} 

\usepackage{xcolor}

\usepackage{amsmath}

\newtheorem{thm}{Theorem}
\newtheorem{lem}[thm]{Lemma}

\newtheorem{cor}[thm]{Corollary}
\newtheorem{rem}{Remark}
\newtheorem{mydef}{Definition}

\usepackage{fancyhdr} 
\usepackage{hyperref}

\title{Existence of cscK metrics on smooth minimal models}
\author{Zakarias Sj\"ostr\"om Dyrefelt}
\address{The Abdus Salam International Centre for Theoretical Physics (ICTP), Str. Costiera, 11, 34151 Trieste TS, Italy. }
\email{zsjostro@ictp.it}


\begin{document}

\begin{abstract}
Given a compact K\"ahler manifold $X$ it is interesting to ask whether it admits a constant scalar curvature K\"ahler (cscK) metric.
In this short note we show that there always exist cscK metrics on compact K\"ahler manifolds with nef canonical bundle, thus on all smooth minimal models, and also on the blowup of any such manifold. 
This confirms an expectation of Jian-Shi-Song \cite{JianShiSong} and extends their main result from $K_X$ semi-ample to $K_X$ nef,
with a direct proof  
that 
does not appeal to the Abundance conjecture. 
As a byproduct we obtain that the connected component $\mathrm{Aut}_0(X)$ of a compact K\"ahler
manifold with $K_X$ nef is 
either trivial or 
a complex torus.
\end{abstract}

\maketitle




\section{Introduction}

\noindent Let $X$ be a compact K\"ahler manifold of dimension $n \geq 2$, let $K_X$ be the canonical bundle and $c_1(X) := c_1(-K_X)$ the associated canonical $(1,1)$-cohomology class. Going back to the work of Calabi \cite{Calabi}, existence of constant scalar curvature (cscK) metrics has been a much studied question in K\"ahler geometry. In particular, by the classical work of Yau \cite{Yau} and Aubin \cite{Aubin} it has been known for decades that cscK metrics always exist on compact K\"ahler manifolds with ample canonical bundle. Much more recently it was shown by Jian-Shi-Song \cite{JianShiSong} that cscK metrics always exist also when $K_X$ is only semi-ample. The step from ample to semi-ample is of special interest thanks to the minimal model program, and especially in view of the Abundance conjecture. Indeed, by MMP and Abundance for surfaces, every compact K\"ahler surface of Kodaira dimension $\kappa \geq 0$ is birational to a surface which is minimal.  
In particular, it follows from \cite{JianShiSong} that all such surfaces admit cscK metrics.
There are some other situations when the Abundance conjecture is known, but in general the gap $K_X$ semi-ample to $K_X$ nef is still conjecture.

In this short note we confirm the expectation of \cite{JianShiSong} that their existence result should hold in general for $-c_1(X)$ nef, with a proof  
not relying on the Abundance conjecture, and 
which is valid for arbitrary (projective or non-projective) compact K\"ahler manifolds in any dimension: 


\begin{thm} \label{Thm main}
Suppose that 
$X$ is a compact K\"ahler manifold with 
$-c_1(X)$ nef. Then for any K\"ahler class $[\omega] \in H^{1,1}(X,\mathbb{R})$, there is  $\epsilon_{X,[\omega]} > 0$, such that for all $0 < \epsilon < \epsilon_{X,[\omega]}$, there exists a unique cscK metric in the K\"ahler class $-c_1(X) + \epsilon[\omega]$. 
In particular, every smooth minimal model admits a cscK metric. 
\end{thm}

\noindent In fact, we prove the stronger statement that the Mabuchi K-energy functional is proper in the above classes $-c_1(X) + \epsilon[\omega]$ 
(see Definition \ref{Definition properness} for explanations of the terminology). Since the K-energy is then invariant under automorphisms, this implies that the automorphism group is compact, and hence the connected component $\mathrm{Aut}_0(X)$ is either a singleton or a complex torus:

\begin{cor} \label{Cor 3}
Suppose that $X$ is a compact K\"ahler manifold with $-c_1(X)$ nef. Then 
$\mathrm{Aut}_0(X)$ is either trivial or a complex torus.
\end{cor}


\noindent This result adds to the literature on automorphism groups for projective algebraic varieties with $K_X$ semi-ample, 
in which case the corresponding result is well-known (see e.g. \cite[Theorem 8.1]{Kobayashibook} and references therein). 


As a further application of Theorem \ref{Thm main} we note that if $-c_1(X)$ is nef, then $X$ cannot support any non-trivial Hamiltionian vector fields, and hence the machinery of Arezzo-Pacard \cite{ArezzoPacard} yields  
existence of cscK metrics also on blowups of smooth minimal models: 

\begin{cor} \label{Cor 2}
Blowups of compact K\"ahler manifolds with  $-c_1(X)$ nef admit cscK metrics. 
\end{cor}




\noindent 
A few comments on the proof of Theorem \ref{Thm main} are in order:
Just as in \cite{JianShiSong} our proof partly relies on the work of Chen-Cheng \cite{ChenChengII}, which is concerned with coercivity of energy functionals on the closure $\mathcal{E}^1$ of the space of K\"ahler metrics. The new ingredient is to combine the machinery in \cite{ChenChengII} with Collins-Sz\'ekelyhidi \cite{CollinsGabor} and an observation about stability thresholds, namely the following numerical quantity 
$$
\delta([\omega]) := \delta_1([\omega]) - \delta_2([\omega])
$$
where
$$
\delta_1([\omega]) := \sup \{ \delta_1 \in \mathbb{R} : \exists C > 0, \mathrm{M}(\varphi) \geq \delta_1||\varphi|| - C, \forall \varphi \in \mathcal{E}^1 \} 
$$
and
$$
\delta_2([\omega]) := \sup \{ \delta_2 \in \mathbb{R} : \exists C > 0, \mathrm{E}_{\omega}^{-\rho_{\omega}}(\varphi) \geq \delta_2||\varphi|| - C, \forall \varphi \in \mathcal{E}^1 \}.
$$
Here $\mathrm{M}$ is the K-energy functional, $\mathrm{E}_{\omega}^{-\rho_{\omega}}$ is the energy part in its Chen-Tian decomposition, see \cite{Chen2000}.
Note that the above thresholds are well-defined real numbers, so in particular they cannot take the value minus infinity (see e.g. \cite[Proposition 12]{SD4}).
By coercivity of entropy we also observe that $\delta([\omega]) \geq \frac{n+1}{n} \alpha_X([\omega]) > 0$ where the latter denotes the classical alpha invariant of Tian \cite{Tian}. 

In order to streamline the proof of Theorem \ref{Thm main} we introduce the following shorthand terminology related to the J-equation and J-stability: 
Let $\theta$ be any smooth $(1,1)$-form on $X$ and write $[\theta] \in H^{1,1}(X,\mathbb{R})$ for the associated cohomology class. We say that $((X,[\omega]); [\theta])$ is \emph{J-coercive} (or \emph{J-stable}, see \cite{CollinsGabor, GaoChen}) if and only if the $E_{\omega}^{\theta}$-functional is coercive on the space $\mathcal{E}^1 := \mathcal{E}^1(X,\omega)$ of finite energy potentials, i.e. there are constants $\delta, C > 0$ such that 
$$
\mathrm{E}_{\omega}^{\theta}(\varphi) \geq ||\varphi|| - C
$$
for all $\varphi \in \mathcal{E}^1$.

\noindent The strategy of proof is then based on the following elementary observation, which characterizes properness of the Mabuchi K-energy via a condition that in view of \cite{SD4} can be interpreted as a geometric condition on the K\"ahler cone:

\begin{thm} \label{Lemma main} 
Suppose that $-c_1(X)$ is nef. Then the K-energy is proper on the space of K\"ahler metrics cohomologous to $\omega$ precisely if $((X,[\omega]); [\theta])$ is J-coercive, where 
$$
[\theta] := -c_1(X) + \delta([\omega])[\omega].
$$
In particular, if either of the above conditions hold, then $X$ admits a cscK metric in the K\"ahler class $[\omega] \in H^{1,1}(X,\mathbb{R})$.
\end{thm}

\noindent Given a compact K\"ahler manifold with $-c_1(X)$ nef, the idea is then to show that the J-coercivity condition of the above theorem is always satisfied for any K\"ahler class close enough to $-c_1(X)$. 
In fact a bit more can be said, as shown by the proofs of Theorem \ref{Thm main} and Theorem \ref{Lemma main} in the next section.
%

\begin{rem}
\emph{By a minor alteration of the proof of Theorem \ref{Thm main} we can more generally obtain a version of the above for the twisted cscK metrics that play a central role in the continuity method of Chen-Cheng \cite{ChenChengII}: Suppose that $\eta \in H^{1,1}(X,\mathbb{R})$ is a $(1,1)$-form such that $-c_1(X) + \eta$ is nef. Then for any K\"ahler class $[\omega]$ on $X$, there is  $\epsilon_{X,[\omega]} > 0$ such that for all $0 < \epsilon < \epsilon_{X,[\omega]}$, there exists a unique $\eta$-twisted cscK metric, that is, a solution to $\mathrm{Tr}_{\beta} (-\mathrm{Ric}(\omega) + \eta) = \mathrm{constant},$ in the K\"ahler class $[\beta] = -c_1(X) + \eta + \epsilon[\omega]$. This follows by simply replacing $-c_1(X)$ by $-c_1(X) + \eta$ everywhere in the proof. }
\end{rem}

\subsection*{Acknowledgements} The author is grateful to Jian Song, Yalong Shi and Claudio Arezzo for helpful comments and discussions, as well as to the anonymous referee for improving the paper, in particular by pointing out Corollary \ref{Cor 3}. 

\medskip

\section{Proof} \label{Section energy functionals} 

\noindent 
We quickly recall the notation and terminology 
for energy functionals 
that we will use, which is the standard variational setup frequently used throughout the K\"ahler geometry literature. 
To introduce our notation, let $(X,\omega)$ be a compact K\"ahler manifold of complex dimension $n \geq 2$ and write $\gamma := [\omega]\in H^{1,1}(X,\mathbb{R})$ for the associated K\"ahler class. 
Let $$V_{\gamma} := \int_X \frac{\omega^n}{n!}$$ be the K\"ahler volume of $(X,\omega)$. 
Let $\rho_{\omega}$ be the Ricci curvature form, normalized such that $[\rho_{\omega}] = c_1(X)$, and write $S(\omega) := \mathrm{Tr}_{\omega}\rho_{\omega}$ for the scalar curvature of $(X,\omega)$. By a standard abuse of terminology we say that $\omega$ is a constant scalar curvature K\"ahler (cscK) metric if $S(\omega)$ is constant.
Denote the automorphism group of $X$ by $\mathrm{Aut}(X)$ and its connected component of the identity by $\mathrm{Aut}_0(X)$. Write $\mathcal{C}_X \subset H^{1,1}(X,\mathbb{R})$ for the cone of K\"ahler cohomology classes on $X$. Let $\overline{\mathcal{C}}_X$ be the nef cone, $\partial \mathcal{C}_X$ its boundary, and let $\mathrm{Big}_X$ be the cone of big $(1,1)$-classes on $X$. 

\noindent We write $(\mathcal{H}_{\omega},d_1)$ for the space of K\"ahler potentials on $X$ endowed with the $L^1$-Finsler metric $d_1$, and 
denote by $(\mathcal{E}^1,d_1)$ its metric completion (see \cite{Darvas14, Darvas15, Darvassurvey, Darvas} and references therein). Write $\mathrm{PSH}(X,\omega) \cap L^{\infty}(X)$ for the space of bounded $\omega$-psh functions on $X$.

Now consider $\varphi \in \mathrm{PSH}(X,\omega) \cap L^{\infty}(X)$.
We may define well-known energy functionals
$$
\mathrm{I}_{\omega}(\varphi) := \frac{1}{V_{\gamma}n!} \int_X \varphi (\omega^n - \omega_{\varphi}^n)
$$ 
$$
\mathrm{J}_{\omega}(\varphi) = \frac{1}{V_{\gamma}n!}\int_X \varphi \omega^n - \frac{1}{V_{\gamma}(n+1)!} \int_X \varphi \sum_{j = 0}^n \omega^j \wedge \omega_{\varphi}^{n-j}
$$
$$
\mathrm{E}_{\omega}^{\theta}(\varphi) := \frac{1}{V_{\gamma}n!} \int_X \varphi \sum_{j = 0}^{n-1} \theta \wedge \omega^j \wedge \omega_{\varphi}^{n-j-1} - \frac{1}{V_{\gamma}(n+1)!}\int_X \varphi \underline{\theta} \sum_{j = 0}^{n} \omega^j \wedge \omega_{\varphi}^{n-j}   
$$
where $\theta$ is any smooth closed $(1,1)$-form on $X$ and $\underline{\theta}$ is the topological constant given by 
$$
\underline{\theta} := \frac{\int_X \theta \wedge \frac{\omega^{n-1}}{(n-1)!}}{\int_X \frac{\omega^n}{n!}}.
$$ By the Chen-Tian formula \cite{Chen2000} the K-energy functional can be written as the sum of an energy/pluripotential part and an entropy part as
\begin{equation} \label{Equation K-energy}
\mathrm{M}_{\omega} = \mathrm{E}_{\omega}^{-\rho_{\omega}} + \mathrm{H}_{\omega}
\end{equation}
where 
$$
\mathrm{H}_{\omega}(\varphi) := \frac{1}{V_{\gamma}n!} \int_X \log \left( \frac{\omega_{\varphi}^n}{\omega^n} \right) \omega_{\varphi}^n
$$
is the relative entropy of the probability measures $\omega_{\varphi}^n/V_{\gamma}$ and $\omega^n/V_{\gamma}$. In particular, it is well known that $\mathrm{H}_{\omega}(\varphi)$ is always non-negative. 

For any given smooth closed $(1,1)$-form $\theta$ on $X$ we also consider the $\theta$-twisted K-energy functional
$$
\mathrm{M}_{\omega}^{\theta} := \mathrm{M}_{\omega} + \mathrm{E}_{\omega}^{\theta}.
$$
As explained in \cite{SD4}, it will also in this paper be convenient to measure properness/coercivity of the K-energy against the 
functional
$$
||\varphi|| := (\mathrm{I}_{\omega} - \mathrm{J}_{\omega})(\varphi) = \frac{1}{V_{\gamma}(n+1)!} \int_X  \varphi \sum_{j=0}^n \omega^j \wedge \omega_{\varphi}^{n-j} - \frac{1}{V_{\gamma}n!} \int_X \varphi \omega_{\varphi}^n
$$
rather than against the Aubin $\mathrm{J}$-functional or the $d_1$-distance. 

\begin{mydef} \label{Definition properness}
Let $\mathrm{F}: \mathcal{E}^1 \rightarrow \mathbb{R}$ be any of the above considered energy functionals. We then say that $\mathrm{F}$ is coercive if 
$$
\mathrm{F}(\varphi) \geq \delta ||\varphi|| - C
$$
for some $\delta, C > 0$ and all $\varphi \in \mathcal{E}^1$.
\end{mydef}




\noindent In the rest of this note we will use the notation of projective polarized manifolds $(X,L)$, while emphasizing that this is only cosmetic, and all the arguments below hold for arbitrary K\"ahler classes on arbitrary (possibly non-projective) compact K\"ahler manifolds.

\subsection*{Proof of Theorem \ref{Lemma main}}
The proof of Theorem \ref{Lemma main} is a simple reformulation of the fact that existence of cscK metrics in a K\"ahler class $c_1(L)$ is connected to J-stability/J-coercivity of the triple $((X,L);K)$ as in the introduction, where $K := K_X$. 
We will use the following necessary and sufficient existence criteria due to \cite{Weinkove1, Weinkove2} and \cite{DonaldsonJobservation}: A polarised manifold $((X,L);K)$ is J-stable 
if 
\begin{equation} \label{Equation W1}
n \frac{K.L^{n-1}}{L^n}L - (n-1)K > 0,
\end{equation}
that is this divisor is ample, and a necessary condition for J-stability of $((X,L);K)$ is that 
\begin{equation} \label{Equation W2}
n \frac{K.L^{n-1}}{L^n}L - K > 0.
\end{equation}
In dimension $n = 2$ the condition thus becomes necessary and sufficient. 

These conditions can equivalently be expressed in terms of stability thresholds, following \cite{SD4}. Indeed, in the notation of that paper we can reformulate the above conditions as a double inequality
\begin{equation} \label{Equation double inequality}
n \frac{K.L^{n-1}}{L^n} - (n-1)\sigma(K,L) \leq 
\delta_{2,K}(L) 
\leq n \frac{K.L^{n-1}}{L^n} - \sigma(K,L),
\end{equation}
where 
$$
\sigma(K,L) := \inf \{\lambda \in \mathbb{R}: K - \lambda L < 0 \}
$$
is a Seshadri type constant and 
$$
\delta_{2,K}(L) 
:= \sup \{\delta \in \mathbb{R} : \exists C > 0, \mathrm{E}_{\omega}^{-\rho_{\omega}}(\varphi) \geq \delta ||\varphi|| - C, \forall \varphi \in \mathcal{E}^1 \}
$$ 
is the optimal coercivity constant for the \emph{energy part} $\mathrm{E}_{\omega}^{-\rho_{\omega}}$ of the Mabuchi K-energy, which by \cite{CollinsGabor} is positive if and only if the triple $((X,L);K)$ is J-stable.

The relationship with the cscK equation follows by considering in the same way the optimal coercivity constant 
$$
\delta_1(L) := \sup \{\delta \in \mathbb{R} : \exists C > 0, \mathrm{M}(\varphi) \geq \delta ||\varphi|| - C, \forall \varphi \in \mathcal{E}^1 \}
$$ 
associated to the K-energy, appealing to the results of \cite{ChenChengII},
and then relating it to J-stability of $((X,L);K)$ by observing that
$$
\delta_1(L) = \delta_{2,K + (\delta_1(L) - \delta_{2,K}(L))L}(L).
$$
In other words the K-energy is proper on the space of K\"ahler metrics in $c_1(L)$ if and only if $((X,L);\theta)$ is J-stable with respect to the smooth 
$(1,1)$-form defined by
\begin{equation*} 
\theta := \theta_{L} := K + (\delta_1(L) - \delta_{2,K}(L))L.
\end{equation*}
Moreover, if any of the above equivalent conditions hold, then $(X,L)$ is a cscK manifold, which is precisely the statement of Theorem \ref{Lemma main}. Note moreover that $\theta$ is not necessarily assumed positive here, since while stability requires positivity to make sense, coercivity can be defined more generally with respect to any smooth $(1,1)$-form $\theta$ on $X$.

\subsection*{Proof of Theorem \ref{Thm main}}

Motivated by Theorem \ref{Lemma main} we now study the $(1,1)$-form $\theta := \theta_L$ more closely. The first thing to comment on is that we cannot expect to compute $\theta_L$ in any easy manner, since then we would know essentially everything about the existence problem for cscK metrics. However, it is immediate to observe the following inequality
\begin{equation} \label{equation inequality}
\delta_1(L) - \delta_{2,K}(L) \; \geq \frac{n+1}{n}\alpha_X(L)  > 0,
\end{equation}
which follows directly from the fact that entropy is coercive (in fact, with the bound $\mathrm{H}(\varphi) \geq \epsilon||\varphi|| - C$ for every $\epsilon < \frac{n+1}{n}\alpha_X(L)$).

Suppose now once and for all that the canonical bundle $K := K_X$ is nef. Then by \eqref{equation inequality} it follows that $\theta$ is ample. 
By adding $\delta_1(L) - \delta_{2,K}(L)$ to each side of the double inequality \eqref{Equation double inequality}, we deduce a similar double inequality also for $\delta_1(L)$:
\begin{equation} \label{Equation double inequality 2}
n \frac{\theta.L^{n-1}}{L^n} - (n-1)\sigma(\theta,L) \leq \delta_1(L) \leq n \frac{\theta.L^{n-1}}{L^n} - \sigma(\theta,L).
\end{equation}
In order to make use of this observation, it will be enough to show that there exists an ample line bundle $L$ on $X$ for which the left hand side of \eqref{Equation double inequality 2} is strictly positive. This is achieved by a simple application of the methods in \cite{SD4}.  
%
%
In case $n=2$ the subcone of the K\"ahler cone where the $\mathrm{E}_{\omega}^{\theta}$-functional is coercive with respect to a given $\theta$ is described explicitly in \cite{SD4}. In higher dimension a complete description was not given, but all that is needed for the proof of our main result are
the criteria \eqref{Equation W1}, \eqref{Equation W2} and the following lemma, whose proof is identical to that in \cite{SD4}: 

\begin{lem} \label{Lemma one over t} Let $\Lambda, T$ be line bundles such that $c_1(\Lambda) \in \mathcal{C}_X$ and $c_1(T) \in \partial \mathcal{C}_X$, that is, $\Lambda$ is ample and $T$ is nef but not ample. 
Then $L_t := (1-t)T + t\Lambda$, $t \in [0,1]$ satisfies
$$
\sigma(\Lambda,L_t) = \frac{1}{t}.
$$
The same result holds if $\Lambda \in \partial \mathcal{C}_X$, as long as $c_1(L_t) \in \mathcal{C}_X$ for all $t \in (0,1)$. 
\end{lem}

\begin{proof}
The first part is \cite[Lemma 16]{SD4}. For the second part, if we write $\sigma'(\beta,\gamma) := \sup \{\delta > 0 : \beta - \delta\gamma > 0\}$ then $\sigma(\Lambda, L_t) = \sigma'(L_t,\Lambda)^{-1}.$ We moreover have $\sigma'(L_t, \Lambda) = (1-t)\sigma'(L_0,\Lambda) + t \sigma'(L_1,\Lambda).$
If $L_0, L_1 \in \partial \mathcal{C}_X$ are boundary classes such that $L_t \in \mathcal{C}_X$ for all $t \in (0,1)$, then $\sigma'(L_0,\Lambda) = 0$ and $\sigma'(L_0,\Lambda) = 1$.  
\end{proof}

\noindent We are finally ready to prove the main result, extending \cite{JianShiSong} and replacing the hypothesis of $K_X$ semi-ample with $K_X$ nef, as predicted by the Abundance conjecture. We recall that while it is here presented in the terminology of projective polarized manifolds $(X,L)$, it is valid also for arbitrary compact K\"ahler manifolds with the same proof.

\begin{thm} \emph{(Theorem \ref{Thm main})}
Suppose that $K_X$ is nef. Then for any ample line bundle $L$ on $X$, there is $\epsilon_{X,L} > 0$, such that for all $0 < \epsilon < \epsilon_{X,L}$, there exists a unique cscK metric in the K\"ahler class $-c_1(X) + \epsilon c_1(L)$. In particular, every smooth minimal model $X$ admits a cscK metric. 
\end{thm}

\begin{proof}
Suppose that $K := K_X$ is nef but not ample (if it is ample then we already know that cscK metrics exist, by Aubin, Yau \cite{Aubin,Yau}). As explained above it is a consequence of \cite{ChenChengII, CollinsGabor, DonaldsonJobservation, Weinkove1, Weinkove2} that cscK metrics must exist if the K\"ahler class
$$
\theta := K + (\delta_1(L) - \delta_{2,K}(L))L 
$$
satisfies
$$
n \frac{\theta.L^{n-1}}{L^n} - (n-1)\sigma(\theta,L) > 0.
$$
Now fix an auxiliary line bundle $T$ which is nef but not ample, in such a way that $L_t := (1-t)T + tK 
$ 
is ample for all $t \in (0,1)$. Consider
$$
R_{\epsilon}(t) := n\frac{(K + \epsilon L_t) \cdot L_t^{n-1}}{L_t^n} - (n-1)\sigma((K + \epsilon L_t), L_t) 
$$
as a function of $\epsilon \in \mathbb{R}$ and $t \in (0,1)$. By the last part of Lemma \ref{Lemma one over t} we then have 
$$
R_{0}(t) = n\frac{K \cdot L_t^{n-1}}{L_t^n} - \frac{n-1}{t},
$$
so that in particular $R_0(1) = 1$. By continuity we may therefore fix $t_0 \in (0,1)$ such that $R_0(t) > 0$ for all $t \in (t_0,1)$. Moreover, it can be easily checked that
$$
\sigma((K + \epsilon L_t), L_t) = \sigma(K,L_t) + \epsilon
$$
and hence also
$$
R_{\epsilon}(t)  = \mathrm{R}(t) + \epsilon.
$$
It follows that $R_{\epsilon}(t) > 0$ for all $t > t_0$ and all $\epsilon > 0$. 
Taking $\epsilon_{L_t} := \delta_1(L_t) - \delta_{2,K}(L_t)$ and noting that this quantity is always positive, it follows from Theorem \ref{Lemma main} and \eqref{Equation double inequality} that $c_1(L_t)$ admits cscK metrics for all $t > t_0$. 
By possibly rescaling $L$ if necessary, it follows that for every ample line bundle $L$ on $X$, there is  $\epsilon_{X,L} > 0$ such that for all $0 < \epsilon < \epsilon_{X,L}$, there exists a unique cscK metric in the K\"ahler class $-c_1(X) + \epsilon c_1(L)$. 
This is what we wanted to prove. 
\end{proof}

\noindent Since it is also clear that the line bundles can be replaced everywhere with arbitrary nef or K\"ahler $(1,1)$-cohomology classes, we obtain precisely the statement of Theorem \ref{Thm main} in the introduction.



%
%
\medskip

\subsection{Proof of Corollary \ref{Cor 3} and Corollary \ref{Cor 2}}

By Theorem \ref{Thm main} there exists a K\"ahler class $[\omega] \in H^{1,1}(X,\mathbb{R})$ (as described in the statement of the theorem) such that the associated Mabuchi K-energy functional $\mathrm{M}$ is $d_1$-proper on the space $\mathcal{H}_{\omega}$ of K\"ahler potentials. 
In particular, by \cite{ChenChengII} there exists a cscK metric in $[\omega]$, and we may fix once and for all a cscK potential $v \in \mathcal{H}_{\omega}$.
Let moreover $G := \mathrm{Aut}_0(X)$ be the connected complex Lie group given by the identity component of the group of biholomorphisms
of $X$. Denote by $\mathfrak{g}$ its Lie algebra, consisting of holomorphic vector fields on $X$. 
The key step of the proof of Corollary \ref{Cor 3} is showing that properness of the K-energy forces $G$ to be compact. To see this, we recall the arguments of \cite{DR}: 
Let $\mathrm{K} := \mathrm{Isom}_0(X,\omega_v)$ be the identity component of the isometry group of $(X,\omega_v)$, and let $\mathfrak{k}$ be 
its Lie algebra. Note that $\mathrm{K}$ is compact, since $X$ is compact (see \cite[Proposition 29.4]{Postnikov}). By a result of Matsushima and Lichnerowicz (see \cite[Theorem 3.6.1]{Gauduchon}) and by \cite[Proposition 6.2]{DR}, it follows that if $g_j \in G$, then there are $k_j \in \mathrm{K}$ and $V_j \in \mathfrak{k}$ such that $g_j = k_j \mathrm{exp}_I(JV_j)$. Recall moreover that $G$ acts on the space of normalized potentials $\mathcal{H}_{\omega} \cap \mathrm{E}^{-1}(0)$ (where $\mathrm{E} := \mathrm{E}_{\omega}^{\omega}$ is the Monge-Amp\`ere energy functional) such that $g.\varphi \in \mathcal{H}_{\omega} \cap \mathrm{E}^{-1}(0)$ is the unique potential defined by the relation
$$
\omega_{g.\varphi} = g^*\omega_{\varphi}.
$$
With this setup, it was shown in \cite[Section 7.1]{DR} that $$[0,+\infty) \ni t \mapsto \mathrm{exp}_I(-tJV_j).v \in \mathcal{H}_{\omega} \cap \mathrm{E}^{-1}(0)$$ is a $d_1$-geodesic ray. 
Since the K-energy is invariant under the action of $G$, properness then implies that $d_1(g_j.v,v) \leq C$, with a bound independent of $j$. Arguing as in \cite[Lemma 3.7]{BDL} it then follows that $d_1(\mathrm{exp}_I(-JV_j).v,v) \leq C$, and hence $||V_j||$ is uniformly bounded in $\mathfrak{k} := \mathfrak{isom}(X,\omega_v)$. By compactness we can then, up to relabeling, choose $V \in \mathfrak{k}$ and $k \in \mathrm{K}$ such that $k_j \rightarrow k$ and $V_j \rightarrow V$ smoothly. Hence also $g_j \rightarrow g := k\mathrm{exp}_I(JV)$ smoothly, and properness implies compactness.
Finally, it is a standard exercise (see e.g. \cite[Chapter 8]{FultonHarris}) to see that a connected compact complex Lie group is abelian, and in fact a complex torus. This proves Corollary \ref{Cor 3}. 

In order to prove Corollary \ref{Cor 2} we finally note that the machinery of Arezzo-Pacard \cite{ArezzoPacard} applies in case there are no non-trivial holomorphic vector fields with zeros, that is no non-trivial Hamiltonian vector fields. To see that we are in this case, we once again use properness of the K-energy functional: Let as before $v$ be a cscK potential and suppose for contradiction that $V \in \mathfrak{isom}(X,\omega_v)$ is a non-trivial Hamiltonian vector field. By the work of \cite{Darvas15} and \cite[Section 7.1]{DR} it follows that $[0,+\infty) \ni t \mapsto \mathrm{exp}_I(-tJV).v$ defines a constant speed geodesic ray in the space of normalized K\"ahler potentials $\mathcal{H}_{\omega} \cap \mathrm{E}^{-1}(0)$.
In particular, if $V$ is non-trivial then well-known convexity properties imply that $d_1(v,\mathrm{exp}_I(-tJV).v) \rightarrow +\infty$ as $t \rightarrow +\infty$ (recall that $d_1$ is comparable to the norm $||.||$ and to Aubin's J-functional, which is strictly convex along geodesics).  
%
On the other hand, since $(X,\omega)$ admits a cscK metric, the K-energy 
is invariant under the action of $G$, contradicting the properness assumption. 
In other words $X$ cannot support any non-trivial Hamiltonian vector fields, and \cite{ArezzoPacard} combined with Theorem \ref{Thm main} yields the statement of 
Corollary \ref{Cor 2}.


\medskip

\begin{thebibliography}{1}

\bibitem{Aubin} T. Aubin, \emph{Equations du type {Monge}-{Amp\`ere} sur les {vari\'et\'es} {k\"ahleriennes} compactes}, Bull. Sci. Math. (2) 102 (1978), no. 1, 63-95


\bibitem{ArezzoPacard} C. Arezzo and F. Pacard, \emph{Blowing up and desingularizing constant scalar curvature {K\"ahler} manifolds}, Acta Math. \textbf{196} (2006), no. 2, 179-228.






\bibitem{BBGZ} R. Berman, S. Boucksom, V. Guedj and A. Zeriahi, \emph{A variational approach to complex {Monge}-{Amp\`ere} equations}, Publ. Math. de l'IHES \textbf{117} (2013), 179-245.


\bibitem{BDL} R. Berman and T. Darvas and C. Lu, \emph{Regularity of weak minimizers of the {K-energy} and applications to properness and {K}-stability}, Preprint arXiv:1602.03114v1, (2016).





\bibitem{Calabi} E. Calabi, \emph{Extremal {K\"ahler} metrics}, In: ``Seminars on Differential Geometry'', S.T. Yau (ed.), Princeton Univ. Press and Univ. of Tokyo Press, Princeton, New York, 1982, 259-290. 




\bibitem{GaoChen} G. Chen, \emph{On {J}-equation}, Preprint arXiv:1905.10222 (2019).


\bibitem{ChenChengII} X.X. Chen and J. Cheng, \emph{On the constant scalar curvature {K\"ahler} metrics, existence results}, Preprint arXiv:1801.00656 (2018).

\bibitem{ChenChengIII} X.X. Chen and J. Cheng, \emph{On the constant scalar curvature {K\"ahler} metrics, general automorphism group}, arXiv:1801.05907 (2018).




\bibitem{Chen2000} X.X. Chen, \emph{On the lower bound of the {Mabuchi} {K-energy} and its application}, Int. Math. Res. Not. \textbf{12} (2000), 607-623.




\bibitem{CollinsGabor} T. Collins and G. Sz\'ekelyhidi, \emph{Convergence of the {J}-flow on toric manifolds}, J. Diff. Geom. \textbf{107} (2017) no. 1, 47-81.

\bibitem{Darvas14} T. Darvas, \emph{The {Mabuchi} completion of the space of {K\"ahler} potentials}, Amer. J. Math \textbf{139} (2017), 1275-1313.

\bibitem{Darvas15} T. Darvas, \emph{The {Mabuchi} Geometry of finite energy classes}, Adv. Math. \textbf{285} (2015), 182-219.

\bibitem{Darvassurvey} T. Darvas, \emph{Geometric pluripotential theory on {K\"ahler} manifolds}, Survey article (2017).

\bibitem{Darvas} T. Darvas, \emph{Weak geodesic rays in the space of {K\"ahler} potentials and the class {$E(X, \omega_0)$}}, J. Inst. Math. Jussieu \textbf{16} (2017), no. 4, 837-858.

\bibitem{DR} T. Darvas and Y.A. Rubinstein, \emph{Tian's properness conjecture and {Finsler} geometry of the space of {K\"ahler} metrics}, J. Amer. Math. Soc. \textbf{30} (2017), no. 2, 347-387.







\bibitem{DonaldsonJobservation} S.K. Donaldson, \emph{Moment maps and diffeomorphisms}, Asian J. Math., \textbf{3} (1999) no. 1, pp. 1-15.

%
%



\bibitem{FultonHarris} W. Fulton and J. Harris, \emph{Representation theory: A first course}, Springer-Verlag New York, 2004. 

\bibitem{Gauduchon} P. Gauduchon, \emph{Calabi's extremal metrics: an elementary introduction}, manuscript, 2010.

\bibitem{JianShiSong} W. Jian, Y. Shi and J. Song, \emph{A remark on constant scalar curvature K\"ahler metrics on minimal models}, Proc. Amer. Math. Soc.  \textbf{147} (2019), 3507-3513.

\bibitem{Kobayashibook} S. Kobayashi, \emph{Transformation groups in differential geometry}, Springer-Verlag Berlin Heidelberg New York, 1972. 





\bibitem{MabuchiKenergy1} T. Mabuchi, \emph{A functional integrating {Futaki} invariants}, Proc. Japan Acad. \textbf{61} (1985), 119-120.

\bibitem{Mabuchi} T. Mabuchi, \emph{K-energy maps integrating {Futaki} invariants}, Tohoku Math. J. \textbf{38} (1986), no. 4, 575-593.

\bibitem{Mabuchisymplectic} T. Mabuchi, \emph{Some symplectic geometry on compact {K\"ahler} manifolds {I}}, Osaka J. Math. \textbf{24} (1987), 227-252.



\bibitem{Postnikov}  M.M. Postnikov, \emph{Geometry VI. Riemannian geometry}, Springer, 2001.








\bibitem{SD4} Z. {Sj\"ostr\"om} Dyrefelt, \emph{Optimal lower bounds for Donaldson's J-functional}, Adv. Math. DOI: 10.1016/j.aim.2020.107271. 




\bibitem{SongWeinkove} J. Song and B. Weinkove, \emph{On the convergence and singularities of the {J}-flow with applications to the {Mabuchi} energy}, Comm. Pure Appl. Math., \textbf{61} (2008), 210 - 229.



\bibitem{Tian} G. Tian, \emph{{K\"ahler}-{Einstein} metrics with positive scalar curvature}, Invent. Math. \textbf{130} (1997), no. 1, 1-37.






\bibitem{Weinkove1} B. Weinkove, \emph{Convergence of the J-flow on {K\"ahler} surfaces}, Comm. Anal. Geom. \textbf{12} (2004), no. 4, 949-965.


\bibitem{Weinkove2} B. Weinkove, \emph{On the J-flow in higher dimensions and the lower boundedness
of the Mabuchi energy}, J. Diff. Geom. \textbf{73} (2006), no. 2, 351-358.

\bibitem{Yau} S.-T. Yau, \emph{On the Ricci curvature of a compact {K\"ahler} manifold and the complex Monge-{Amp\`ere} equation, I, Comm. Pure Appl. Math. 31 (1978), 339 - 411.}





\end{thebibliography}
\end{document}